\documentclass[11pt]{article}
\usepackage[a4paper,margin=1in]{geometry}
\usepackage{amsmath,amsthm,amssymb,mathtools}
\usepackage{microtype}
\usepackage{newtxtext,newtxmath}
\usepackage{booktabs}
\usepackage{float}
\usepackage{caption}
\usepackage{tablefootnote}
\usepackage{xcolor}
\usepackage{listings}
\usepackage{hyperref}
\usepackage{bookmark}
\title{Judicious partitions for restricted self-sumsets in cyclic groups}

\author{Keane Maverick}
\date{September 2025}

\newtheorem{theorem}{Theorem}
\newtheorem{lemma}{Lemma}[section]
\newtheorem{corollary}{Corollary}[section]

\newtheorem{remark}{Remark}[section]

\begin{document}

\maketitle

\begin{abstract}
We study the minimax problem for restricted two-fold self-sumsets in $k$-colorings of $\mathbb{Z}_n$. For primes $p$ with $2\le k\le p$ we determine the exact minimum $\max\{0,\,2\lceil p/k\rceil-3\}$. For general $n$ (with $m=\lceil n/k\rceil$) we bound the optimum between a size term $\min\{p(n),\,2m-3\}$ and a periodicity term $f(n/q(n,k))$, and show these bounds are tight when $2m-3\le p(n)$ or $f(n/q(n,k))\le \min\{p(n),\,2m-3\}$. We further prove a stability inequality and a threshold theorem that force concentration in a single subgroup coset near the periodic scale. In the prime case with $m\ge5$ and $2m-3<p$, every optimal coloring contains a class of size $m$ that is an arc (an arithmetic progression up to an affine automorphism). Our approach combines the restricted Erd\H{o}s--Heilbronn phenomenon with block/coset colorings and an injectivity window.


\end{abstract}

\section{Introduction}

Judicious partition problems ask how well one can optimize a worst-part statistic over all $k$-colorings (for related work, see \cite{BollobasScott1999}).
In additive combinatorics, a natural statistic of ``arithmetical richness'' of a set $A$ is the size of its sumset.
Here we measure richness by the restricted two-fold self-sumset $A\widehat{+}A$
(only sums of two distinct elements of $A$), and ask:

\medskip
Among all $k$-colorings of $\mathbb{Z}_n$, how small can the largest restricted self-sumset be?
\medskip

Formally, we minimize over partitions $\mathbb{Z}_n=A_1\sqcup\cdots\sqcup A_k$
the quantity $\max_i |A_i\widehat{+}A_i|$. The resulting extremal value is $\widehat{\Phi}_k(n)$.

Two ``effects'' determine the scale of the minimum. 
First, a size effect: in any $k$-coloring there is a class of size 
$m=\lceil n/k\rceil$, and for such a set $A$ one has 
$|A\widehat{+}A|\ge \min\{p(n), 2m-3\}$ \cite{Karolyi2004, DiasdaSilvaHamidoune1994}.
Second, a periodicity effect: if $A\subseteq a+H$ for a subgroup $H\le \mathbb{Z}_n$ of size $n/q$, then $|A\widehat{+}A|\le f(n/q)$,
with equality when $A=a+H$.
Our results identify the minimum by comparing these two quantities and give 
stability statements near the periodic scale.

\paragraph{Why restricted sums (and not unrestricted).}
For unrestricted sums $A + A$, the exact minimum $\min_{|A| = m}|A + A|$ in finite abelian groups is known \cite{EliahouKervairePlagne2003}, making the partition minimax a quick corollary.
The restricted setting (distinct summands) over composite modulus does not admit such a closed form.

\paragraph{Organization.}
Section~\ref{sec:prelims} fixes notation and basic objects (2.1) and states the main theorems (2.2).
Section~\ref{sec:proofs} proves the prime case, establishes the general bounds, and pinpoints the exact regimes.
Section~\ref{sec:stability} develops the stability inequality and the threshold theorem.

\section{Preliminaries}\label{sec:prelims}

\subsection{Definitions}

We work in the additive cyclic group $\mathbb{Z}_n=\{0,1,\dots,n-1\}$ with addition taken modulo $n$.
A \emph{$k$-coloring} (partition) of $\mathbb{Z}_n$ is a disjoint union
\[
\mathbb{Z}_n = A_1 \sqcup A_2 \sqcup \cdots \sqcup A_k,
\]
where the sets $A_i$ are the color classes. In every $k$-coloring there is a class of size at least
\[
m = m(n,k) := \left\lceil \frac{n}{k}\right\rceil .
\]

For a set $A\subseteq \mathbb{Z}_n$, the restricted self-sumset collects sums of two distinct elements of $A$:
\[
A\widehat{+}A := \{ a+a'\pmod{n} : a,a'\in A,\ a\neq a'\}.
\]
By convention $|A \widehat{+} A| = 0$ if $|A|\le 1$.

Our objective is the minimax value
\[
\widehat{\Phi}_k(n) := \min_{\mathbb{Z}_n = A_1\sqcup\cdots\sqcup A_k}\ \max_{1\le i\le k}\ |A_i\widehat{+}A_i|,
\]
the smallest possible value of the largest restricted self-sumset among the color classes.

Two arithmetic parameters will be used repeatedly. First, for any integer $r\ge2$, let $p(r)$ denote the least prime divisor of $r$. In particular $p(n)$ is the least prime divisor of $n$.
Second, $q(n,k)$ is the largest divisor of $n$ that is at most $k$ (if no nontrivial divisor is $\le k$, set $q(n,k) = 1$).
We also write, for $t\in\mathbb{Z}_{\ge1}$,
\[
f(t):=
\begin{cases}
0,& t=1,\\[2pt]
1,& t=2,\\[2pt]
t,& t\ge 3,
\end{cases}
\]
a function that matches the restricted self-sumset size of a full coset of a subgroup of size $t$.

It will be convenient to speak about arcs and blocks.
Viewing $\mathbb{Z}_n$ as the circle $\{0,1,\dots,n-1\}$ in cyclic order, an arc is a set of consecutive residues (possibly wrapping around $n-1$ to $0$).
A block is a consecutive set that does not wrap (i.e., of the form $\{s,s+1,\dots,s+t-1\}\ \pmod{n}$).
If $I$ is a block of size $t\ge 2$ and $2t-3\le n$, then the distinct sums with different summands run through an interval of $2t-3$ integers. Hence,
\[
|I \widehat{+} I| = 2t-3 .
\]

Cosets and quotients will play a central role.
If $H\le \mathbb{Z}_n$ is a subgroup of index $q$ (so $|H|=n/q$) and $a\in\mathbb{Z}_n$, then the coset $a+H$ has restricted self-sumset
\[
|(a+H) \widehat{+} (a+H)| = f(|H|) = f(n/q),
\]
and any subset $B\subseteq a + H$ satisfies $B\widehat{+}B \subseteq 2a+H$ (so its restricted sums remain in the same coset).
Given $H$, let $\pi_H:\mathbb{Z}_n\to \mathbb{Z}_n/H$ be the quotient map. We define
\[
r_H(A) := |\pi_H(A)| \quad\text{and}\quad
\sigma_H(A):=\big|\{\text{$H$-cosets }C : |A\cap C|\ge 2\}\big|
\]
as the number of $H$-cosets that $A$ meets, and the number of $H$-cosets where $A$ has at least two elements.

We write $\operatorname{diam}(X):=\max(X)-\min(X)$ for the diameter of a finite $X\subset\mathbb{Z}$, computed in the integers.

Finally, we record the standard restricted Erd\H{o}s--Heilbronn lower bound in this setting:
for all $A\subseteq\mathbb{Z}_n$ with $|A|\ge 2$,
\[
|A \widehat{+} A| \ge \min\{p(n),\ 2|A|-3\},
\]
see \cite{Karolyi2004}. For the prime case see \cite{DiasdaSilvaHamidoune1994}.
This is consistent with our convention $|A \widehat{+} A| = 0$ when $|A|\le 1$.

\subsection{The theorems we show}\label{sec:main}

The proofs are deferred to later sections. Here we record the statements that the paper establishes.

\begin{theorem}\label{thm:prime}
Let $p$ be prime and $2\le k\le p$. Then,
\[
\widehat{\Phi}_k(p) = \max\big\{0,\ 2\left\lceil \tfrac{p}{k}\right\rceil - 3\big\}.
\]
\end{theorem}

\begin{theorem}\label{thm:prime-ext}
Let $p$ be prime and $2\le k\le p$, and put $m=\lceil p/k\rceil$. If $2m-3<p$ and $m\ge5$, then in every optimal $k$-coloring of $\mathbb{Z}_p$ attaining $\widehat{\Phi}_k(p)=2m-3$, there exists a color class $A$ of size $m$ that is an arc (equivalently, an arithmetic progression up to an affine automorphism $x\mapsto ux+v$ with $u\in\mathbb{Z}_p^\times$).
\end{theorem}

\begin{theorem}\label{thm:bounds}
For all $n\ge2$ and $k\ge2$, with $m=\left\lceil \tfrac{n}{k}\right\rceil$,
\[
\max\big\{0,\ \min\{p(n),2m-3\}\big\} \le \widehat{\Phi}_k(n) \le \min\big\{\max(0,2m-3),\ f\!\big(\tfrac{n}{q(n,k)}\big)\big\}.
\]
\end{theorem}

\begin{theorem}\label{thm:interval-exact}
If $2\left\lceil \tfrac{n}{k}\right\rceil - 3 \le p(n)$, then
\[
\widehat{\Phi}_k(n) = \max\big\{0,2\left\lceil \tfrac{n}{k}\right\rceil-3\big\}.
\]
\end{theorem}

\begin{theorem}\label{thm:coset-exact}
If $f\!\big(\tfrac{n}{q(n,k)}\big) \le \min\big\{p(n),\ 2\left\lceil \tfrac{n}{k}\right\rceil-3\big\}$, then
\[
\widehat{\Phi}_k(n) = f\!\big(\tfrac{n}{q(n,k)}\big).
\]
\end{theorem}

\begin{remark}
In all other $(n,k)$ the value $\widehat{\Phi}_k(n)$ is not claimed to be exact. By Theorem~\ref{thm:bounds} it lies between the stated lower and upper bounds.
\end{remark}

\begin{theorem}\label{thm:stability-coset-ineq}
Let $H\le \mathbb{Z}_n$ have size $t\ge 3$, and let $A\subseteq \mathbb{Z}_n$.
Choose a coset $C=a+H$ maximizing $x:=|A\cap C|$, and set $r:=r_H(A)-1$ (the number of occupied $H$-cosets other than $C$).
Define
\[
\alpha^\star:=\max\{0,\ \min\{p(t),\ 2x-3\}\}.
\]
Then,
\[
\big|A\widehat{+}A\big| \ge \alpha^\star + rx.
\]
\end{theorem}

\begin{theorem}\label{thm:stability-coset-threshold}
In the setting of Theorem~\ref{thm:stability-coset-ineq}, suppose $\big|A \widehat{+} A\big|\le t+s$ for some integer $s\ge 0$, where $t=|H|$.
Let $x=|A\cap C|$ for a heaviest coset $C$, and define $\alpha^\star:=\max\{0,\ \min\{p(t),\ 2x-3\}\}$.
Then
\[
r \le \frac{t+s-\alpha^\star}{x}.
\]
In particular, if $2x-3\le p(t)$ and $3x>t+s+3$, then $r=0$ and hence $A\subseteq C$.
\end{theorem}

\begin{remark}[Edge cases $|A|\in\{0,1,2\}$ and $|H|\in\{1,2\}$]
Our conventions give $|A\widehat{+}A|=0$ for $|A|\le1$ and $|A \widehat{+} A|=1$ for $|A|=2$.
For subgroups of size $1$ or $2$ we use $f(1)=0$ and $f(2)=1$, and the periodic statements adapt with these values.
\end{remark}

\section{Bounds and exact regimes}\label{sec:proofs}

\subsection{Lemmas for Section \ref{sec:proofs}}
We collect the elementary tools we use, and the two standard restricted-sumset lower bounds.

\begin{lemma}\label{lem:block}
Let $I\subseteq \mathbb{Z}_n$ be a non-wrapping block of consecutive residues of size $t\ge 0$.
Choose representatives $I^\ast=\{s,s+1,\dots,s+t-1\}\subset\mathbb{Z}$.
Then the set of integer sums with distinct summands
\[
S:=\{x+y:\ x,y\in I^\ast,\ x\neq y\}
\]
has cardinality $|S|=\max\{0,2t-3\}$.
Reducing modulo $n$ yields
\[
|I\widehat{+}I|\ \le\ \max\{0,2t-3\},
\]
with equality if and only if $2t-3\le n$.
\end{lemma}

\begin{proof}
If $t\le 1$, there are no distinct pairs, so the integer count is $0$.
For $t\ge 2$, writing $I^\ast=\{s,s+1,\dots,s+t-1\}$ in the integers, the distinct sums $x+y$ with $x\neq y\in I^\ast$
range over the contiguous interval $\{2s+1,2s+2,\dots,2s+2t-3\}$, giving exactly $2t-3$ values. Hence, $|S|=2t-3$.
The image of $S$ modulo $n$ is precisely $I\widehat{+}I$, so $|I\widehat{+}I|\le |S|=\max\{0,2t-3\}$.
If $2t-3\le n$, then the interval $\{2s+1,\dots,2s+2t-3\}$ has diameter $2t-4\le n-1$, so no two distinct elements can differ by $n$, and reduction modulo $n$ is injective on $S$, yielding equality.
Conversely, if $2t-3\ge n+1$ then the diameter is at least $n$, and the interval contains two elements differing by $n$, forcing a collision and strict inequality. Thus, equality holds if and only if $2t-3\le n$.
\end{proof}

\begin{lemma}\label{lem:coset}
Let $H\le \mathbb{Z}_n$ be a subgroup of size $t$, and let $a\in \mathbb{Z}_n$.
Then
\[
|(a+H)\widehat{+}(a+H)| = f(t)
\quad\text{where}\quad
f(t)=\begin{cases}
0,&t=1,\\
1,&t=2,\\
t,&t\ge 3.
\end{cases}
\]
Additionally, for any $B\subseteq a+H$, we have $|B\widehat{+}B|\le |(a+H)\widehat{+}(a+H)|=f(t)$.
\end{lemma}

\begin{proof}
If $t=1$, the coset has one element and there are no distinct pairs.
If $t=2$, the coset is $\{a,a+h\}$ with $h\neq 0$ of order $2$. Hence, the only distinct sum is $2a+h$, so size is $1$.
If $t\ge 3$, fix $x\in H$. Choose $u\in H$ with $u\ne x-u$ (possible since $|H|\ge 3$), then
$x= u+(x-u)$ is a sum of two distinct elements of $H$. Hence, $x\in H\widehat{+}H$.
This implies $H\widehat{+}H=H$, and by translation $(a+H)\widehat{+}(a+H)=2a+H$, so the size is $t$.
The subset claim is immediate, since $B \widehat{+} B\subseteq (a+H) \widehat{+} (a+H)$.
\end{proof}

\begin{lemma}\label{lem:ddsh}
Let $p$ be prime and $A\subseteq \mathbb{Z}_p$ with $|A|\ge 2$. Then
\[
|A \widehat{+} A| \ge \min\{p,\ 2|A|-3\}.
\]
This is the Dias da Silva--Hamidoune bound \cite{DiasdaSilvaHamidoune1994}.
\end{lemma}

\begin{lemma}\label{lem:karolyi}
Let $n\ge 2$ and $A\subseteq \mathbb{Z}_n$ with $|A|\ge 2$.
Then,
\[
|A \widehat{+} A| \ge \min\{p(n),\ 2|A|-3\}.
\]
This bound was proved by K{\'a}rolyi \cite{Karolyi2004}.
\end{lemma}

\begin{lemma}\label{lem:inverse-reh}
Let \(p\) be prime and \(A\subseteq\mathbb{Z}_p\) with \(|A|=m\) and \(2m-3<p\).
If \(|A\widehat{+}A|=2m-3\), then \(A\) is an arithmetic progression (equivalently, an arc up to an affine automorphism \(x\mapsto ux+v\) with \(u\in\mathbb{Z}_p^\times\)).
This is the inverse (equality) case due to K{\'a}rolyi \cite{Karolyi2005Inverse}.
\end{lemma}

\begin{remark}
Lemmas~\ref{lem:ddsh} \cite{DiasdaSilvaHamidoune1994}, \ref{lem:karolyi} \cite{Karolyi2004}, and \ref{lem:inverse-reh} \cite{Karolyi2005Inverse} are used as black boxes. 
For an alternative proof of Lemma~\ref{lem:ddsh} over \(\mathbb{Z}_p\) via the polynomial method, see \cite{AlonNathansonRuzsa1996}.
For \(|A|\le 1\) the bounds are consistent with \(|A\widehat{+}A|=0\).
\end{remark}

\begin{lemma}\label{lem:injective-window}
Let $n\ge2$. Let $S\subset \mathbb{Z}$ be a set of integers contained in an interval of length $<p(n)$.
Then reduction modulo $n$ is injective on $S$. In particular, if $X\subset \mathbb{Z}$ is any finite set with $\operatorname{diam}(X)<p(n)$, the map $X\to \mathbb{Z}_n$ has no collisions.
\end{lemma}

\begin{proof}
Assume, for the sake of contradiction, that there exist distinct $x,y\in S$ with $x\equiv y\pmod{n}$.
Then $n\mid(x-y)$. Let $p=p(n)$ be the least prime divisor of $n$.
Since $p \mid n$, we also have $p \mid (x-y)$. Hence, $|x-y|\ge p$.
However, $S$ is contained in an interval of length $<p$, so for distinct $x,y\in S$ we have $0<|x-y|<p$, this is a contradiction.
Therefore the only possibility is $x=y$. Thus, the reduction map is injective on $S$.
\end{proof}

\begin{lemma}\label{lem:largest-block-placement}
Let $n,k\ge2$ and set $m=\left\lceil \frac{n}{k}\right\rceil$. There exists a partition of $\{0,1,\dots,n-1\}$ into $k$ consecutive, non-wrapping blocks whose sizes differ by at most $1$, and whose largest block has size $m$.
\end{lemma}

\begin{proof}
Write $n=ak+b$ with $0\le b<k$. Take $b$ blocks of length $a+1$ followed by $k-b$ blocks of length $a$, in the linear order $0,1,\dots,n-1$, none wraps. Hence, the largest block has size $a+1=\lceil n/k\rceil=m$.
\end{proof}

\subsection{Proof of Theorem \ref{thm:prime}}

\begin{proof}
Let $p$ be prime and $2\le k\le p$, and write $m=\left\lceil \tfrac{p}{k}\right\rceil$.
If $m=1$ (equivalently $k=p$), then every color class has size at most $1$, so $|A_i\widehat{+}A_i|=0$ for all $i$ and hence $\widehat{\Phi}_k(p)=0=\max\{0,2m-3\}$. 

Assume $m\ge2$. In any $k$-partition of $\mathbb{Z}_p$, some color class $A$ satisfies $|A|\ge m$. By Lemma~\ref{lem:ddsh},
\[
|A \widehat{+} A| \ge \min\{p,\ 2|A|-3\} \ge \min\{p,\ 2m-3\}.
\]
Since $k\ge2$, we have $m\le \lceil p/2\rceil$, so $2m-3\le p-2<p$. Hence, $\min\{p,2m-3\}=2m-3$. Thus, we have the lower bound $\widehat{\Phi}_k(p)\ge 2m-3=\max\{0,2m-3\}$.

For the matching construction, partition $\{0,1,\dots,p-1\}$ into $k$ consecutive, non-wrapping blocks, whose sizes differ by at most one. Let $I$ be a largest block, so $|I|=m$ (Lemma~\ref{lem:largest-block-placement}). By Lemma~\ref{lem:block},
\[
|I\widehat{+}I|=\max\{0,2m-3\},
\]
and all other blocks have size $m$ or $m-1$. Hence, $|A_i\widehat{+}A_i|\le \max\{0,2m-3\}$ for each $i$. Therefore, we have the upper bound
$\widehat{\Phi}_k(p)\le \max\{0,2m-3\}$.

Combining the two bounds gives $\widehat{\Phi}_k(p)=\max\{0,2\lceil p/k\rceil-3\}$.
\end{proof}

\subsection{Proof of Theorem \ref{thm:prime-ext}}\label{subsec:prime-extremizers}

\begin{proof}
Let $p$ be prime and $m=\lceil p/k\rceil$. Assume $2m-3<p$ and $m\ge5$.
By Theorem~\ref{thm:prime}, there exists an optimal $k$-coloring whose value is
\(\widehat{\Phi}_k(p)=2m-3\). Fix such an optimal coloring and let its color
classes be $A_1,\dots,A_k$.

First, no class can have size $\ge m+1$. Assume, for the sake of contradiction, there exists an $A_i$ where $|A_i|\ge m+1$.
Then, by Lemma~\ref{lem:ddsh},
\[
|A_i \widehat{+} A_i| \ge \min\{p,\,2|A_i|-3\} \ge 2(m+1)-3 = 2m-1 > 2m-3,
\]
so the maximum over classes would exceed $2m-3$, contradicting optimality of the
coloring.

Hence, every class has size $\le m$. Since $\sum_i |A_i|=p$ and
$k(m-1)<p$ (because $m=\lceil p/k\rceil$), at least one class must have size
exactly $m$. For this class (call it $A$), the optimality of the coloring forces
\(|A\widehat{+}A|\le 2m-3\), and Lemma \ref{lem:ddsh} gives \(|A\widehat{+}A|\ge 2m-3\), hence
\(|A\widehat{+}A|=2m-3\).

Finally, since $2|A|-3=2m-3<p$ and $|A|=m\ge 5$, Lemma~\ref{lem:inverse-reh} implies that $A$ is an arc (equivalently, an arithmetic progression up to an affine automorphism $x\mapsto ux+v$ with $u\in\mathbb{Z}_p^\times$).
\end{proof}

\begin{corollary}\label{cor:prime-block}
For prime $p$ and $2\le k\le p$, there exists an optimal coloring attaining $\widehat{\Phi}_k(p)$ in which the $k$ color classes are consecutive, non-wrapping blocks, whose sizes differ by at most $1$. In particular, the largest block has size $m=\lceil p/k\rceil$.
\end{corollary}

\begin{proof}
Apply Lemma~\ref{lem:largest-block-placement} to partition $\{0,1,\dots,p-1\}$ into $k$ consecutive, non-wrapping blocks, whose sizes differ by at most $1$, and take these $k$ blocks as the color classes. Let $I$ be a largest block, so $|I|=m=\lceil p/k\rceil$. By Lemma~\ref{lem:block},
\[
|I\widehat{+}I|=\max\{0,2m-3\}.
\]
Since $k\ge2$, we have $m\le \lceil p/2\rceil$, hence $2m-3\le p-2<p$. Thus $|I\widehat{+}I|=2m-3$ if $m\ge2$, and $|I\widehat{+}I|=0$ if $m=1$. Every other block has size $m$ or $m-1$, so for each color class $A_i$,
\[
|A_i\widehat{+}A_i|\le \max\{0,2m-3\}.
\]
Therefore the maximum over colors is $\max\{0,2m-3\}$. By Theorem~\ref{thm:prime}, $\widehat{\Phi}_k(p)=\max\{0,2m-3\}$. The constructed block coloring attains $\widehat{\Phi}_k(p)$ and has the stated structure.
\end{proof}

\begin{corollary}\label{cor:prime-arc}
Under the hypotheses of Theorem~\ref{thm:prime-ext}, every optimal coloring admits at least one color class of size $m$ that is an arc (equivalently, an arithmetic progression up to an affine automorphism $x\mapsto ux+v$ with $u\in\mathbb{Z}_p^\times$).
\end{corollary}

\begin{proof}
In an optimal coloring under $2m-3<p$ and $m\ge5$, some class must have size exactly $m$ and attain $|A\widehat{+}A|=2m-3$.
By Lemma~\ref{lem:inverse-reh}, that class is an arithmetic progression with nonzero difference.
Equivalently, after an affine automorphism $x\mapsto ux+v$ ($u\in\mathbb{Z}_p^\times$), it is an arc.
\end{proof}

\subsection{Proof of Theorem \ref{thm:bounds}}

\begin{proof}
Fix $n\ge 2$ and $k\ge 2$, and set $m=\left\lceil \tfrac{n}{k}\right\rceil$.
In any $k$-partition of $\mathbb{Z}_n$, some color class $A$ has $|A|\ge m$.
Applying Lemma~\ref{lem:karolyi} to this $A$ yields
\[
|A\widehat{+}A| \ge \min\{p(n),\ 2|A|-3\}
 \ge \min\{p(n),\ 2m-3\}.
\]
As $|A\widehat{+}A|\ge 0$ always, we obtain the lower bound
\[
\widehat{\Phi}_k(n) \ge \max\big\{0,\ \min\{p(n),2m-3\}\big\}.
\]

To attain the upper bound, we use two constructions.

\smallskip
First, block coloring.
Split $\{0,1,\dots,n-1\}$ into $k$ consecutive, non-wrapping blocks whose sizes differ by at most $1$.
The largest block has size $m$, and every other has size $m$ or $m-1$.
By Lemma~\ref{lem:block}, each block $I$ satisfies $|I\widehat{+}I|\le \max\{0,2|I|-3\}\le \max\{0,2m-3\}$.
Therefore,
\[
\max_{1\le i\le k}|A_i\widehat{+}A_i| \le \max\{0,2m-3\}.
\]

\smallskip
Second, coset coloring.
Let $q=q(n,k)$ be the largest divisor of $n$ with $q\le k$.
Choose a subgroup $H\le \mathbb{Z}_n$ with index $q$ (so $|H|=n/q$),
and color each coset $a+H$ with a different color.
If $k>q$, split one or more cosets into additional colors. Every new color is still contained in some coset.
By Lemma~\ref{lem:coset}, each full coset $a+H$ has
$|(a+H) \widehat{+} (a+H)| = f(|H|) = f(n/q)$, and any subset of a coset has restricted sums contained in the same coset,
thus never exceeding $f(n/q)$. Hence,
\[
\max_{1\le i\le k}|A_i\widehat{+}A_i| \le f\!\big(\tfrac{n}{q(n,k)}\big).
\]

\smallskip
Taking the better (smaller) of the two constructions gives
\[
\widehat{\Phi}_k(n) \le \min\big\{\max(0,2m-3),\ f\!\big(\tfrac{n}{q(n,k)}\big)\big\}.
\]
Together with the lower bound, this proves the theorem.
\end{proof}

\subsection{Proof of Theorem \ref{thm:interval-exact}}

\begin{proof}
Assume $2\left\lceil \tfrac{n}{k}\right\rceil-3 \le p(n)$ and set $m=\left\lceil \tfrac{n}{k}\right\rceil$.
In any $k$-partition of $\mathbb{Z}_n$, some color class $A$ has $|A|\ge m$. By the lower bound in Theorem \ref{thm:bounds},
\[
\max_i |A_i\widehat{+}A_i| \ge \max\big\{0,\ \min\{p(n),2m-3\}\big\}
 = \max\{0,2m-3\},
\]
since $2m-3\le p(n)$ by hypothesis.

By Lemma \ref{lem:largest-block-placement}, partition $\{0,1,\dots,n-1\}$ into $k$ consecutive, non-wrapping blocks with largest block $I$ of size $m$.
By Lemma \ref{lem:block}, for a representative interval $I^\ast$ of the largest block $I$ (with $|I|=m$), the integer sums with distinct summands form a contiguous interval of length $2m-3$ (after translation). This integer interval has diameter $2m-4$. Since $2m-3\le p(n)$, we have $2m-4<p(n)$. Hence, reduction modulo $n$ is injective by Lemma~\ref{lem:injective-window}, and therefore
\[
|I \widehat{+} I| = 2m-3 = \max\{0,2m-3\}.
\]
All other blocks have size $m$ or $m-1$, so by Lemma \ref{lem:block} their restricted self-sumsets have size $\le \max\{0,2m-3\}$.
Thus, the constructed coloring satisfies
\[
\max_{1\le i\le k}|A_i \widehat{+} A_i| \le \max\{0,2m-3\}.
\]

Combining the lower and upper bounds yields
\(
\widehat{\Phi}_k(n) = \max\{0,2\lceil n/k\rceil-3\}.
\)
\end{proof}

\medskip

\subsection{Proof of Theorem \ref{thm:coset-exact}}

\begin{proof}
Assume
\[
f\!\big(\tfrac{n}{q(n,k)}\big) \le \min\Big\{p(n),\ 2\big\lceil \tfrac{n}{k}\big\rceil-3\Big\},
\]
and set $q=q(n,k)$ and $t=\tfrac{n}{q}$.

Choose a subgroup $H\le \mathbb{Z}_n$ of index $q$ (so $|H|=t$), and color each coset $a+H$ with its own color. If $k>q$, split some cosets further (staying within cosets).
By Lemma \ref{lem:coset}, every full coset satisfies
\(
|(a+H) \widehat{+} (a+H)| = f(t) = f\!\big(\tfrac{n}{q(n,k)}\big),
\)
and any subset of a coset has restricted sums contained in the same coset, hence never exceeding $f(t)$.
Therefore, we arrive at the upper bound
\[
\widehat{\Phi}_k(n) \le f\!\big(\tfrac{n}{q(n,k)}\big).
\]

By Theorem \ref{thm:bounds}, we can write the lower bound
\[
\widehat{\Phi}_k(n) \ge \max\big\{0,\ \min\{p(n),2\lceil n/k\rceil-3\}\big\}
 \ge f\!\big(\tfrac{n}{q(n,k)}\big),
\]
where the last inequality is precisely our regime assumption.
Thus the equality 
\(
\widehat{\Phi}_k(n) = f\!\big(\tfrac{n}{q(n,k)}\big)
\) 
holds.

We note the edge cases $t\in\{1,2\}$.
If $t=1$ then $f(t)=0$ and necessarily $k\ge q=n$, so $m=\lceil n/k\rceil=1$ and the regime condition holds. The value $0$ is achieved because every color has size $\le1$.
If $t=2$, then $f(t)=1$. Once again, the regime condition implies $1\le \min\{p(n),2m-3\}$ (in particular $m\ge2$), and the coset coloring across the index-$2$ subgroup attains value $1$.
\end{proof}

\section{Stability}\label{sec:stability}

This section proves the stability statements recorded in Section~\ref{sec:main}.
We begin with a short background paragraph, then collect three lemmas we will use, and finally give the proofs of
Theorems~\ref{thm:stability-coset-ineq} and \ref{thm:stability-coset-threshold}.

\subsection{Background}

When a color class $A$ in $\mathbb{Z}_n$ is \emph{periodic} (mostly contained in a coset $a+H$ of a subgroup $H$),
its restricted self-sumset $A \widehat{+} A$ is constrained to live almost entirely inside the coset $2a+H$, whose size is $|H|$ when $|H|\ge 3$.
Thus, values of $|A \widehat{+} A|$ close to $|H|$ indicate strong concentration of $A$ in a single $H$-coset.
Our stability results quantify this: the heaviest coset forces many cross-coset sums that cannot overlap with the within-coset sums,
yielding a clean inequality and a threshold theorem.

\subsection{Lemmas for Section \ref{sec:stability}}

Throughout, $H \le \mathbb{Z}_n$ is a subgroup of size $t\ge 1$,
$C=a+H$ denotes a coset, and $p(t)$ denotes the least prime divisor of $t$.

\begin{lemma}\label{lem:REH-in-H}
Identify $H$ with $\mathbb{Z}_t$ via an additive isomorphism.
For any $B\subseteq C$ with $|B|\ge 2$,
\[
|B \widehat{+} B| \ge \min\{p(t),\ 2|B|-3\},
\]
and $B\widehat{+}B\subseteq 2C$.
\end{lemma}

\begin{proof}
If $t=1$ then the premise $|B|\ge 2$ cannot hold, so the claim is vacuous. 
For $t\ge2$, translation by $-a$ identifies $C$ with $H$ and $B$ with a subset of $H\cong \mathbb{Z}_t$.
Applying Lemma~\ref{lem:karolyi} in the group $\mathbb{Z}_t$ yields
\(|B\widehat{+}B|\ge \min\{p(t),\,2|B|-3\}\).
Translating back shows $B\widehat{+}B\subseteq 2C$.
\end{proof}

\begin{remark}
For a fixed $H$, the quantity $\alpha(C):=\max\{0,\min\{p(t),2|A\cap C|-3\}\}$ is positive only on those $H$-cosets $C$ with $|A\cap C|\ge2$, whose number is $\sigma_H(A)$. Thus, any sum of $\alpha(C)$’s effectively ranges over exactly $\sigma_H(A)$ cosets.
\end{remark}

\begin{lemma}\label{lem:distinct-cross}
Let $C=a+H$ and let $D=b+H$ and $D'=b'+H$ be cosets of $H$.
Then, $(D+C)=(D'+C)$ if and only if $D=D'$.
Equivalently, for fixed $C$, the map $D\mapsto D+C$ is a bijection on the set of $H$-cosets.
\end{lemma}

\begin{proof}
$(D+C)=(b+H)+(a+H)=(a+b)+H$ depends only on the class $b+H$. Distinct classes give distinct sums in the quotient $\mathbb{Z}_n/H$.
\end{proof}

\begin{lemma}\label{lem:multi-coset-internal}
Let $H\le\mathbb{Z}_n$ have size $t\ge3$ and index $q_H=[\mathbb{Z}_n\!:\!H]$.
For each $H$-coset $C$, put $A_C:=A\cap C$ and
\[
\alpha(C):=\max\{0,\min\{p(t),\,2|A_C|-3\}\}.
\]
If $q_H$ is odd, the map $C\mapsto 2C$ is a bijection on $H$-cosets, and
\[
|A\widehat{+}A| \ge \sum_{\substack{C\\|A_C|\ge2}}\alpha(C).
\]
In general,
\[
|A\widehat{+}A| \ge \sum_{E} \max_{\substack{C:\,2C=E}}\alpha(C),
\]
where the sum runs over the $H$-cosets $E$ of the form $E=2C$.
\end{lemma}

\begin{proof}
By Lemma~\ref{lem:REH-in-H}, $(A_C\widehat{+}A_C)\subseteq 2C$ and
$|A_C\widehat{+}A_C|\ge\alpha(C)$. If $q_H$ is odd then the sets $2C$ are distinct,
so the internal restricted sums from different $C$ lie in disjoint cosets, and the sum of sizes applies.
When $q_H$ is even, group the occupied cosets by their image $2C=E$: in each $E$ the union of
internal sums has size at least the largest $\alpha(C)$ in that fiber. Summing over $E$ gives the claim.
\end{proof}

\subsection{Proof of Theorem \ref{thm:stability-coset-ineq}}

\begin{proof}
Let $H\le \mathbb{Z}_n$ have size $t\ge 3$, and let $A\subseteq \mathbb{Z}_n$.
Choose a coset $C=a+H$ maximizing $x:=|A\cap C|$, and let $r$ be the number of other $H$-cosets meeting $A$.
Set
\[
\alpha^\star := \max\{0,\ \min\{p(t),\ 2x-3\}\}.
\]

Write $A_C:=A\cap C$ and $x=|A_C|$.
If $x\le 1$ then $\alpha^\star=0$, and for each other occupied coset $D=b+H$ choose any $y\in A\cap D$.
Translation by $y$ is injective, so the $x$ sums $y+A_C$ are pairwise distinct and lie in $D+C$.
As $D$ ranges over the $r=r_H(A)-1$ distinct cosets $D\ne C$, Lemma~\ref{lem:distinct-cross} gives pairwise distinct cosets $D+C$,
so these $rx$ sums are all distinct. Hence, $|A\widehat{+}A|\ge rx=\alpha^\star+rx$ as claimed.

Assume now $x\ge 2$. By Lemma~\ref{lem:REH-in-H},
\[
|A_C\widehat{+}A_C| \ge \min\{p(t),\ 2x-3\}\quad\text{and}\quad A_C\widehat{+}A_C\subseteq 2C.
\]
For each other occupied coset $D=b+H$, fix $y\in A\cap D$.
Then $y+A_C$ contributes exactly $x$ distinct residues lying in $D+C$. Since $D\ne C$, we have $D+C\ne 2C$.
As $D$ varies over the $r=r_H(A)-1$ distinct cosets, Lemma~\ref{lem:distinct-cross} implies these cosets $D+C$ are pairwise distinct.
Hence, the $rx$ sums are distinct and disjoint from $2C$.
Adding the $\min\{p(t),2x-3\}$ sums inside $2C$ yields at least
$\min\{p(t),2x-3\}+rx=\alpha^\star+rx$ distinct residues in $A\widehat{+}A$.
\end{proof}

\begin{remark}\label{rem:two-routes-max}
Combining Theorem~\ref{thm:stability-coset-ineq} with Lemma~\ref{lem:multi-coset-internal} yields
\[
|A\widehat{+}A| \ge \max\Big\{\ \alpha^\star+rx,\ \ \sum_{E}\max_{C:\,2C=E}\alpha(C)\ \Big\}.
\]
Here $r=r_H(A)-1$, $\alpha(C):=\max\{0,\min\{p(t),\,2|A\cap C|-3\}\}$, and $q_H:=|\mathbb{Z}_n/H|=n/t$. Also, if $x=|A\cap C_\star|$ for a heaviest coset $C_\star$, then the number $s$ of cosets with $|A_C|=x$ satisfies $1\le s\le \sigma_H(A)$.

This internal-sums bound is sometimes sharper. For example, if $q_H$ is odd and
$A$ meets exactly $s\ge2$ cosets with $|A_C|=x\ge3$, and if $p(t)\ge 2x-3$ so no capping occurs, then
\[
\sum_{E}\max_{C:\,2C=E}\alpha(C)=s(2x-3)
\quad\text{while}\quad
\alpha^\star+rx=(2x-3)+(s-1)x=xs+(x-3),
\]
where we used $r=r_H(A)-1=s-1$ in this configuration. Thus, the internal-sums bound exceeds $\alpha^\star+rx$ by $(s-1)(x-3)$ whenever $x\ge4$ (equal when $x=3$). If $q_H$ is even or $p(t)$ is small, the advantage may vanish due to collisions or capping.
\end{remark}

\subsection{Proof of Theorem \ref{thm:stability-coset-threshold}}
\begin{proof}
In the setting of Theorem~\ref{thm:stability-coset-ineq}, assume additionally that
\[
\big|A \widehat{+} A\big| \le t+s\quad\text{for some integer }s\ge 0,
\]
where $t=|H|$.
Let $x=|A\cap C|$ for a heaviest coset $C$, and set $\alpha^\star:=\max\{0,\min\{p(t),2x-3\}\}$.
By Theorem~\ref{thm:stability-coset-ineq},
\(
|A \widehat{+} A| \ge \alpha^\star + rx.
\)
Combining with the assumed upper bound gives $rx\le t+s-\alpha^\star$. Hence,
\[
r \le \frac{t+s-\alpha^\star}{x}\quad\text{for $x>0$. If $x=0$, then $A=\varnothing$ and the inequality is trivial.}
\]

If $2x-3\le p(t)$ then $\alpha^\star=2x-3$, and the inequality $3x>t+s+3$ implies
\(
\frac{t+s - (2x-3)}{x} < 1
\),
so $r<1$. By integrality, $r=0$ and hence $A\subseteq C$.
\end{proof}

\bibliographystyle{amsplain}
\bibliography{refs}

\end{document}